\documentclass[reqno]{amsart}
\usepackage{microtype}

\usepackage[shortlabels]{enumitem}
\setlist[enumerate]{label={\arabic*.}}

\usepackage{amssymb}
\usepackage{amsmath}
\usepackage[dvipsnames]{xcolor}
\usepackage[colorlinks=true,linkcolor=Maroon,citecolor=OliveGreen,backref]{hyperref}
\usepackage{url}
\usepackage[abbrev,alphabetic]{amsrefs}  % bibliography
\usepackage{cleveref, autonum}

\newtheorem{theorem}{Theorem}

\newtheorem{proposition}[theorem]{Proposition}
\newtheorem{corollary}[theorem]{Corollary}

\theoremstyle{definition}

\newtheorem{conjecture}[theorem]{Conjecture}
\newtheorem{remark}[theorem]{Remark}
\newtheorem{construction}{Construction}\crefname{construction}{construction}{constructions}

\numberwithin{theorem}{section}
\numberwithin{equation}{section}

\newcommand\br[1]{{\left(#1\right)}}
\newcommand\floor[1]{\left\lfloor{#1}\right\rfloor}

\renewcommand{\tilde}{\widetilde}

\def\P{\mathbf{P}}

\newcommand\opr[1]{\operatorname{#1}}

\def\GL{\opr{GL}}
\def\GammaL{\opr{\Gamma L}} %chktex 1
\def\PGL{\opr{PGL}}
\def\PSL{\opr{PSL}}
\def\PGammaL{\opr{P\Gamma L}} %chktex 1
\def\Gal{\opr{Gal}}

\def\Aut{\opr{Aut}}

\def\chr{\opr{char}}
\def\tr{\opr{tr}}
\def\dev{\opr{dev}}

\newcommand{\FF}[0]{\mathbb{F}}
\def\eps{\varepsilon}

\newcommand{\Ell}{\mathcal{L}}
\def\calP{\mathcal{P}}
\def\F{\mathbf{F}}
\def\Z{\mathbf{Z}}
\def\R{\mathbf{R}}
\def\C{\mathbf{C}}
\def\cP{\mathcal{P}}

\def\cO{\mathcal{O}}

\def\frob{\opr{Frob}}
\def\spn{\opr{span}}
\renewcommand{\matrix}[9]{\begin{pmatrix} #1 & #2 & #3 \\ #4 & #5 & #6 \\ #7 & #8 & #9 \end{pmatrix}}
\newcommand{\trigonal}[6]{\matrix%
    {#1}{#4}{#6}%
    {0}{#2}{#5}%
    {0}{0}{#3}}
\newcommand{\unipotent}[3]{\trigonal111{#1}{#2}{#3}}
\newcommand{\diagonal}[3]{\trigonal{#1}{#2}{#3}000}

\begin{document}

\title{The apparent structure of dense Sidon sets}
%alternative: Dense Sidon sets and projective planes
%or (Ben's suggestion): A bestiary of dense Sidon sets

\author{Sean Eberhard}

\address{Sean Eberhard, Mathematical Sciences Research Centre, Queen's University Belfast, Belfast BT7~1NN, UK}
\email{s.eberhard@qub.ac.uk}

\thanks{SE has received funding from the European Research Council (ERC) under the European Union’s Horizon 2020 research and innovation programme (grant agreement No. 803711) and from the Royal Society.}

\author{Freddie Manners}
\address{Freddie Manners, UCSD Department of Mathematics, 9500 Gilman Drive \#0112, La Jolla CA 92093, USA}
\email{fmanners@ucsd.edu}

\begin{abstract}
    The correspondence between perfect difference sets and transitive projective planes is well-known.
    We observe that all known dense (i.e., close to square-root size) Sidon subsets of abelian groups come from projective planes through a similar construction.
    We classify the Sidon sets arising in this manner from desarguesian planes and find essentially no new examples.
    There are many further examples arising from nondesarguesian planes.

    We conjecture that all dense Sidon sets arise from finite projective planes in this way.
    If true, this implies that all abelian groups of most orders do not have dense Sidon subsets.
    In particular if $\sigma_n$ denotes the size of the largest Sidon subset of $\Z/n\Z$, this implies $\liminf_{n \to \infty} \sigma_n / n^{1/2} < 1$.

    We also give a brief bestiary of somewhat smaller Sidon sets with a variety of algebraic origins, and for some of them provide an overarching pattern.
\end{abstract}

\maketitle

\setcounter{tocdepth}{1}
\tableofcontents

\section{Dense Sidon sets}%
\label{sec:intro}

Let $G$ be an abelian group.
A \emph{Sidon set} (or \emph{$B_2$ set}) is a subset $S \subset G$ such that
\[ 
    x + y = z + w \implies \{x, y\} = \{z, w\} \qquad (x, y, z, w \in S).
\]
We call a solution $(x, y, z, w)$ to $x + y = z + w$ an \emph{additive quadruple},
and a \emph{trivial additive quadruple} if $\{x, y\} = \{z, w\}$, so a Sidon set is a set all of whose additive quadruples are trivial.
We call $S$ a \emph{perfect difference set} if moreover $S - S = G$.

Sidon sets are interesting in additive combinatorics for being extremely unstructured: they have maximum doubling constant and minimum additive energy (see~\cite{tao--vu}*{Chapter~2} for the definitions of these terms).
It is curious therefore that all known Sidon sets which are nearly as large as possible appear to be rather structured in some other way (while for instance Sidon sets constructed randomly or greedily have much smaller size).
As Ruzsa put it, ``somehow all known constructions of dense Sidon sets involve the primes'' (\cite{ruzsa-survey}*{Section~11}).

Let $S$ be a Sidon set in an abelian group $G$ of order $n$.
Since the differences $x-y$ with $x \neq y$ are all distinct and nontrivial, we must have $|S|(|S|-1) \leq n-1$, or $|S|^2 - |S| + 1 \leq n$.
Call $S$ \emph{dense} if $|S| \geq (1 - o(1)) n^{1/2}$.
The following are the best-known examples of dense Sidon sets.

\begin{construction}[Erd\H{o}s--Tur\'an~\cite{erdos--turan}]%
    \label{csn:1}
    We give this example first, slightly out of chronological order, because it is the simplest to describe and understand.
    Let $K$ be a finite field.
    Assume $\chr K \neq 2$.
    Let $G = K^2$.
    Let $S$ be the parabola
    \[
        S = \{(x, x^2) : x \in K\}.
    \]
    
    It is a simple exercise to check the Sidon property.
    Suppose
    \[
        (x, x^2), (y, y^2), (z, z^2), (w, w^2)
    \]
    form an additive quadruple.
    Then
    \begin{align}
        x + y &= z + w \\
        x^2 + y^2 &= z^2 + w^2.
    \end{align}
    Hence  
    \[
        2xy = (x + y)^2 - (x^2 + y^2) = (z + w)^2 - (z^2 + w^2) = 2zw.
    \]
    Since $\chr K \neq 2$, $xy = zw$.
    Hence the polynomials $(t - x)(t-y)$ and $(t - z)(t - w)$ are equal, so $\{x, y\} = \{z, w\}$.
    
    \emph{Parameters:} $|G| = q^2$ and $|S| = q$, where $q = |K|$.
\end{construction}

\begin{construction}[Singer~\cite{singer}]%
    \label{csn:2}
    Let $K$ be a finite field and let $L$ be an extension of $K$ of degree $3$.
    Let $G = L^\times / K^\times$.
    Let $H$ be a $K$-plane in $L$, say
    \[
        H = \{x \in L : \tr x = 0\}.
    \]
    Let $S = (H \cap L^\times) / K^\times$.
    
    To check the Sidon property, suppose $x, y, z, w \in S$ form an additive quadruple.
    Let $\tilde x, \tilde y, \tilde z, \tilde w$ be lifts to $L^\times$.
    Then $\tilde x \tilde y \tilde z^{-1} \tilde w^{-1} \in K^\times$.
    Let
    \[
        H' = \tilde x \tilde z^{-1} H = \tilde w \tilde y^{-1} H.
    \]
    Note that $H$ and $H'$ both contain both $\tilde x$ and $\tilde w$.
    Hence either $H = H'$ or $\tilde x / \tilde w \in K^\times$.
    In other words, either $x = z$ or $x = w$, as required.
    
    \emph{Parameters:} $|G| = q^2 + q + 1$ and $|S| = q + 1$, where $q = |K|$ ($S$ is a perfect difference set).
\end{construction}

\begin{construction}[Bose~\cite{bose}]%
    \label{csn:3}
    Let $K$ be a finite field and let $L$ be an extension of $K$ of degree $2$.
    Let $G = L^\times$.
    Let $H$ be a $K$-line in $L$, let $u \in L \setminus H$, and let $S = u + H$.
    
    The verification of the Sidon property is much as in \Cref{csn:2}.
    
    \emph{Parameters:} $|G| = q^2 - 1$ and $|S| = q$, where $q = |K|$.
\end{construction}

\begin{construction}[Spence: see Ganley~\cite{ganley}, Ruzsa~\cite{ruzsa}*{Theorem~4.4}]%
    \label{csn:4}
    Let $K$ be a finite field, let $G = K^\times \times K$, and let $S = \{(x, x) : x \in K^\times\}$.
    
    If $(x,x), (y, y), (z, z), (w, w)$ form an additive quadruple then $x + y = z + w$ and $xy = zw$,
    so $\{x, y\} = \{z, w\}$, as in \Cref{csn:1}.
    
    \emph{Parameters:} $|G| = q(q-1)$ and $|S| = q - 1$, where $q = |K|$.
\end{construction}

\begin{construction}[Hughes~\cite{hughes}, Cilleruelo~\cite{cilleruelo}*{Example~3}]%
    \label{csn:5}
    Let $K$ be a finite field, let $G = K^\times \times K^\times$, and let $S = \{(x, y) : x, y \neq 0, x + y = 1\}$.
    
    Suppose $(x, 1-x), (y, 1-y), (z, 1-z), (w, 1-w)$ form an additive quadruple.
    Then
    \begin{align}
        xy &= zw \\
        (1-x)(1-y) &= (1-z)(1-w).
    \end{align}
    We deduce $x + y = z + w$, and it follows that $\{x, y\} = \{z, w\}$ as in \Cref{csn:1,csn:4}.
    
    \emph{Parameters:} $|G| = (q - 1)^2$ and $|S| = q-2$, where $q = |K|$.
\end{construction}

In the literature there is particular emphasis on Sidon sets in cyclic groups, since those may be used to define Sidon sets in $\Z$.
The groups in \Cref{csn:2,csn:3} are cyclic, the groups in \Cref{csn:1,csn:5} are not, and the group in \Cref{csn:4} is cyclic if and only if $q$ is prime.
In this paper we are equally interested in all abelian groups.

Our first main observation is  that the five constructions presented above are not as varied as they appear.
In fact there is a correspondence with the largest abelian subgroups of $\PGL_3(K)$.
The correspondence associates to each abelian subgroup $G \leq \PGL_3(K)$ the Sidon set given as the point-line stabilizer
\[
    S = \{ g \in G: p^g \in \ell \},
\]
for some point $p$ and line $\ell$ in the projective plane $\P^2(K)$ such that the stabilizers $G_p$ and $G_\ell$ are trivial.
There are essentially no further examples in this correspondence.
All this is covered in Section~\ref{sec:correspondence} and Section~\ref{sec:desarg}.

On the other hand, the correspondence is valid for arbitrary finite projective planes, desarguesian\footnote{A projective plane is called \emph{desarguesian} if it satisfies Desargues's theorem. Desarguesian finite projective planes are exactly those of the form $\P^2(K)$ for some finite field $K$.} or not.
The plane should have a large abelian group of automorphisms, playing the role of $G \le \PGL_3(K)$ in the desarguesian case, which significantly constrains the projective planes to be considered.
Nevertheless, many further examples arise in this way; see Section~\ref{sec:nondesarg}.

%We outline the correspondence in general in Section~\ref{sec:correspondence}, and handle the desarguesian case in detail in Section~\ref{sec:desarg}.

Conversely, the correspondence also shows that any dense Sidon set gives rise to an object that is ``almost'' a projective plane.
We cannot show, but it is natural to conjecture, that these objects are always true projective planes with some points and lines missing, meaning that all dense Sidon sets are obtained from projective planes.
We will state some precise conjectures of this form in Section~\ref{sec:conjectures}.

Assuming this conjecture, known results on projective planes having large abelian automorphism groups (specifically, the Dembowski--Piper classification) limits which abelian groups $G$ could possibly admit dense Sidon sets $S \subseteq G$.
In particular, the conjecture would imply that
$
    \liminf_{n \to \infty} \max \bigl\{ |S| \colon S \subseteq \Z/n\Z~\text{Sidon} \bigr\} / n^{1/2} < 1
$.

Even Sidon sets which are significantly smaller than $(1-o(1)) n^{1/2}$, e.g., by a constant factor or a power of $\log n$, seem to be rather structured, although the situation is less rigid.
In Section~\ref{sec:smaller} we gather some apparently varied existing constructions, and show---similarly to the dense case above---that many of them fit a common pattern. 
We also use this general pattern to generate new examples.

\subsection{Notation}
We adopt the group-theoretic conventions that, for a group $G$ acting on a set $X$ and $g \in G$, $x \in X$, $x^g$ denotes the action of $g$ on $x$ and $G_x$ denotes the stabilizer of $x$ in $G$.
We will also use standard big $O$ and little $o$ notation occasionally (as we have done already), as well as the Vinogradov notation $X \ll Y$ to mean $X = O(Y)$.

\section{The correspondence}%
\label{sec:correspondence}

An \emph{incidence structure} $\Ell$ is abstractly just a triple $(P, L, I)$ of sets such that $I \subset P \times L$ (this is also the abstract definition of a bipartite graph).
Conventionally we call the elements of $P$ \emph{points}, the elements of $L$ \emph{lines}, and the elements of $I$ \emph{incidences}.
We say that $p \in P$ and $\ell \in L$ are \emph{incident} if $(p, \ell) \in I$, and we may write $p \in \ell$.
We freely use further geometric language: we say $\ell$ and $\ell'$ intersect if there is a point $p$ incident to both of them, we say $p$ and $p'$ are joined by a line if there is a line incident to both of them, etc.
The following definitions, of increasing specialization, are standard.
\begin{enumerate}
    \item An incidence structure $\Ell$ is a \emph{partial linear space} if any two distinct points are incident with at most one line.
    (Equivalently, the bipartite graph defined by $I \subset P \times L$ contains no $C_4$.)
    \item A partial linear space $\Ell$ is
    \begin{enumerate}[(a)]
        \item a \emph{linear space} if any two points are joined by a line,
        \item a \emph{dual linear space}\footnote{sometimes a \emph{semiplane} or a \emph{partial projective plane}} if any two lines intersect.
    \end{enumerate}
    \item A partial linear space which is both a linear space and a dual linear space is a \emph{projective plane}.
\end{enumerate}

%A partial linear space is called \emph{nondegenerate} if lines are determined by the points they contain, and vice versa. This is equivalent to assuming that every line contains at least two points. Thus we can identify lines with subsets of $P$ in such a way that $\in$ is consistent with the set-theoretic definition.

A \emph{collineation} or \emph{morphism} $\phi$ between incidence structures $\Ell = (P, L, I)$ and $\Ell' = (P', L', I')$ is a pair of maps $P\to P'$ and $L \to L'$ (both denoted $\phi$) such that
\[
    p \in \ell \implies p^\phi \in \ell^\phi.
\]
As usual, an \emph{isomorphism} is a morphism with an inverse morphism.
We write $\Aut \Ell$ for the group of automorphisms of $\Ell$.

Now let $G$ be an abelian group.
The following proposition articulates a basic equivalence between
Sidon sets $S \subset G$ and partial linear spaces $\Ell$ with a regular $G$-action.

\begin{proposition}%
    \label{prop:correspondence}
    Suppose $\Ell$ is a partial linear space and $G$ is an abelian subgroup of $\Aut \Ell$ such that the action of $G$ is regular on both points and lines.
    Then for any point $p$ and line $\ell$, the set
    \[
        S = \{ g \in G : p^g \in \ell \}
    \]
    is a Sidon set in $G$.
    
    Conversely, suppose $G$ is an abelian group and $S \subset G$.
    The \emph{development} $\dev(S)$ of $S$ is the incidence structure $(P, L, I)$ with $P = L = G$ and
    \[
        I = \{(p, \ell) \in G^2 : p - \ell \in S\}.
    \]
    The incidence structure $\dev(S)$ is a partial linear space if (and only if) $S$ is a Sidon set.
    Every point is contained in $|S|$ lines and every line contains $|S|$ points,
    and $G$ acts regularly on both points and lines.
\end{proposition}
\begin{proof}
    For the first part, suppose $x, y, z, w \in S$ and $x z^{-1} = w y^{-1}$.
    Let
    \[
        q = p^{xz^{-1}} = p^{w y^{-1}}.
    \]
    Then
    \[
        p, q \in \ell^{z^{-1}}, \ell^{y^{-1}}.
    \]
    Since $\Ell$ is a partial linear space, this implies $p = q$ or $\ell^{z^{-1}} = \ell^{y^{-1}}$.
    Since the action is regular, this implies $x = z$ or $y = z$.
    Hence $S$ is a Sidon set, as claimed.
    
    For the second part, let $\Ell = \dev(S)$, and suppose $p_1, p_2 \in \ell_1, \ell_2$.
    Then $p_1 - \ell_1, p_1 - \ell_2, p_2 - \ell_1, p_2 - \ell_2 \in S$.
    Since
    \[
    (p_1 - \ell_1) + (p_2 - \ell_2) = (p_1 - \ell_2) + (p_2 - \ell_1),
    \]
    the Sidon condition implies that $p_1 = p_2$ or $\ell_1 = \ell_2$.
    Hence $\Ell$ is a partial linear space.
    The further claims are clear.
\end{proof}

\begin{remark}
    The dual $\Ell^*$ of an incidence structure $\Ell$ is the incidence structure $(L, P, I^*)$, where $I^* = \{(\ell, p) : (p, \ell) \in I\}$.
    An incidence structure $\Ell$ is \emph{self-dual} if $\Ell \cong \Ell^*$.
    The development $\dev(S)$ of a set $S \subset G$ is always self-dual: the dual incidence set is $\dev(-S)$, and the maps $P \to L^\ast$, $x \mapsto -x$ and $L \to P^\ast$, $x \mapsto -x$ define an isomorphism.
\end{remark}

\begin{remark}
    \Cref{prop:correspondence} is well-established in the design theory literature in the extreme case of perfect difference sets and projective planes,
    but less well-known in the general case.
    The reason is more cultural than mathematical: design theorists are not interested in Sidon sets beyond the cases of difference sets and relative difference sets,
    while additive-combinatorialists are interested in quite sparse Sidon sets (anything above cube-root density, usually in $\{1, \dots, n\}$) and the development $\dev(S)$ is less interesting in that case.
\end{remark}

Recall that for any projective plane $\calP$ there is a positive integer $q$, called the \emph{order} of $\calP$, such that there are $q^2 + q + 1$ points, $q^2 + q + 1$ lines, and the incidence graph is $(q+1)$-regular.

\begin{corollary}%
    \label{cor:restricted-projective-planes}
    Let $\mathcal{P}$ be a projective plane of order $q$, and let $G$ be an abelian subgroup of $\Aut\mathcal{P}$.
    Let $p$ be a point and $\ell$ a line such that $G_p = G_\ell = 1$,
    and suppose $\ell$ contains $d$ points of $P \setminus Gp$.
    Then $S = \{g \in G: p^g \in \ell\}$
    is a Sidon set of size $q+1-d$,
    and
    \begin{equation}
        \label{eq:dbound}
        d \leq (q+1) \br{ \frac{q^2+q+1}{|G|} - 1 }.
    \end{equation}
\end{corollary}
\begin{proof}
    A partial linear space with a regular $G$-action is obtained by restricting to the orbits of $p$ and $\ell$,
    so the fact that $S$ is a Sidon set follows from \Cref{prop:correspondence}.
    Since $\ell$ contains $q+1-d$ points of $Gp$ and the action of $G$ on $Gp$ is regular, it is clear that $|S| = q + 1 - d$.
    We must prove the bound on $d$.
    
    Consider the bipartite incidence graph between $P \setminus Gp$ and $G\ell$.
    The degree of each vertex in $G\ell$ is $d$, while the degree of each vertex in $P \setminus Gp$ is at most $q+1$, so by counting edges we have
    \[
        d|G| \leq (q+1) (q^2 + q + 1 - |G|).
    \]
    Rearranging gives~\eqref{eq:dbound}.
\end{proof}

Note that \Cref{cor:restricted-projective-planes} guarantees a dense Sidon set if and only if $|G| = (1 - o(1))( q^2 + q + 1)$.
In other words, almost all the points of $\calP$ must be in a single $G$-orbit and the same for the lines.
We will see that this is a harsh restriction.

\section{Desarguesian constructions}%
\label{sec:desarg}

The desarguesian projective plane $\P^2(K)$ over the finite field $K$ is defined by taking the points and lines to be the lines and planes, respectively, in the three-dimensional vector space $K^3$, with incidence defined naturally.
In this section we establish that the five constructions of Sidon sets listed in the introduction
arise from \Cref{cor:restricted-projective-planes} applied to $\P^2(K)$ and the maximal abelian subgroups of $\PGL_3(K)$.%
\footnote{It was previously observed by Tait and Timmons~\cite{tait-timmons} that the Cayley graph of the Bose Sidon set (Construction~\ref{csn:3}) is a large subset of an ``orthogonal polarity graph'', i.e., $\P^2(K)$ with the points and lines identified by a self-duality. This is essentially a special case of this general correspondence.  In unpublished work, Timmons made similar observations about the Erd\H{o}s--Tur\'an and Spence examples (Constructions~\ref{csn:1} and~\ref{csn:4}) (Michael Tait, personal communication).}
No further examples arise in this way (apart from a variant of \Cref{csn:1} in even characteristic).

\begin{proposition}%
    \label{prop:maximal-abelian}
    Let $K$ be the finite field $\F_q$.
    Let $Z \cong K^\times$ be the center of $\GL_3(K)$.
    The maximal abelian subgroups of $\PGL_3(K) = \GL_3(K) / Z$, are, up to conjugacy:
    \begin{enumerate}[label=(\roman*)]
        \item
        identifying $K^3$ with the field $L = \F_{q^3}$, the cyclic group $L^\times / K^\times$;
        \item
        identifying $K^3$ with $L \times K$ where $L = \F_{q^2}$, the group $L^\times K^\times / K^\times$;
        \item the group $(K^\times)^3 / Z \cong (K^\times)^2$ of diagonal matrices mod $Z$;
        \item
        \[
            \left\{ \trigonal{r}{r}{1}{a}00 : r \in K^\times,\, a \in K \right\} Z/Z \cong K^\times \times K;
        \]
        \item
        \[
            \left\{ \unipotent{a}{a}{b} : a,b \in K \right\} Z/Z \cong
            \begin{cases}
                K^2 &: q~\textup{odd},\\
                C_4^d &: q = 2^d;
            \end{cases}
        \]
        \item
        \[
            \left\{ \unipotent{0}{a}{b} : a,b \in K \right\} Z/Z \cong K^2;
        \]
        \item
        \[
            \left\{ \unipotent{a}{0}{b} : a,b \in K \right\} Z/Z \cong K^2,
        \]
        \item (only if $q \equiv 1 \pmod 3$) the group $\langle g, h\rangle Z/Z \cong C_3 \times C_3$, where
        \[
            g = \diagonal{1}{\omega}{\omega^2},
            \qquad
            h = \matrix%
            001%
            100%
            010
            ,
        \]
        where $\omega \in K$ is a primitive cube root of unity;
        \item (only if $q \equiv 1 \pmod 3$) identifying $K^3$ with $L = \F_{q^3}$, the group
        \[
            \langle(K^\times)^{1/3}, \frob_q\rangle K^\times / K^\times \cong C_3 \times C_3,
        \]
        where $\frob_q$ is the $x \mapsto x^q$ automorphism of $L / K$.
    \end{enumerate}
\end{proposition}
\begin{proof}
    It is equivalent to classify maximal subgroups $G \leq \GL_3(K)$ such that $Z \leq G$ and $G' \leq Z$.
    
    First suppose $G' = 1$.
    Let $g \in G$.
    If $h \in G$ then $h$ preserves the generalized eigenspaces of $g$ over $\bar K$, the algebraic closure of $K$.
    If $g$ has an eigenvalue of degree $3$ then we are in case \emph{(i)}, while if $g$ has an eigenvalue of degree $2$ then we are in case \emph{(ii)}.
    Hence we may assume all eigenvalues of all $g \in G$ are in $K$.
    
    Let $m$ be maximum number of distinct eigenvalues of any $g \in G$, so $m \in \{1, 2, 3\}$.
    If $m = 3$ then we are in case \emph{(iii)}.
    Suppose $m = 2$.
    Then some $g \in G$ has two distinct eigenvalues, one with algebraic multiplicity 2,
    so $G$ preserves a decomposion of the form $K^3 = U \oplus W$ where $\dim U = 2$ and $\dim W = 1$.
    By maximality, $G$ is the direct product of its projections to $\GL(U) \cong \GL_2(K)$ and $\GL(W) \cong K^\times$.
    In the $\GL_2(K)$ factor, each element must have the form $\text{scalar}\times\text{unipotent}$, so we can assume $G$ is upper-triangular.\footnote{Unipotent subgroups can be upper-triangularized: see~\cite{wehrfritz}*{Corollary~1.21}.  Alternatively, it is elementary that commuting sets of matrices can be upper-triangularized (over any field containing all their eigenvalues), by finding a common eigenvector $e_1$ and using induction on the quotient by $\langle e_1 \rangle$.}
    This is case \emph{(iv)}.
    
    Hence assume $m = 1$, so all $g \in G$ have the form $\text{scalar} \times \text{unipotent}$, and again we may choose a basis in which $G$ is upper-triangular.
    Since $G$ is maximal, $G = ZU$ for some unipotent subgroup
    \[
        U \leq \unipotent***.
    \]
    Since $U$ is abelian, we must have
    \[
        U = \left\{
            \unipotent{tx_0}{ty_0}{z}
            : t, z \in K
        \right\} \qquad (x_0, y_0 \in K).
    \]
    By conjugating by a diagonal matrix can assume $x_0, y_0 \in \{0, 1\}$, and by maximality $(x_0, y_0) \in \{(1, 1), (0, 1), (1, 0)\}$, giving cases \emph{(v)}, \emph{(vi)}, and \emph{(vii)}, respectively.
    
    Now suppose $G'$ is a nontrivial subgroup of $Z$.
    Pick $g, h \in G$ such that $[g, h]$ is some nontrivial scalar $\omega \in K$.
    Then $h$ permutes the generalized eigenspaces of $g$ and the eigenvalues according to $\lambda \mapsto \omega \lambda$,
    so we must have $\omega^3 = 1$ and $g$ has eigenvalues $\lambda, \omega\lambda, \omega^2\lambda$ for some $\lambda \in \bar K$.
    The determinant of $g$ is $\lambda^3$, so $\lambda \in (K^\times)^{1/3}$.
    Similarly, $h$ has eigenvalues $\mu, \omega\mu, \omega^2\mu$ for some $\mu \in (K^\times)^{1/3}$.
    
    We claim that $G = \langle g, h \rangle Z$.
    Suppose $x \in G$. Then $[x, g], [x, h] \in Z$, so $[x, g] = \omega^i$ and $[x, h] = \omega^j$ for some $i, j \in \{0, 1, 2\}$.
    Let $y = x g^{-j} h^i$. Then $[y, g] = [y, h] = 1$.
    But the centralizer of $\{g, h\}$ is $Z$, so $y \in Z$, so $x \in \langle g, h \rangle Z$.
    
    Suppose $\lambda, \mu \in K$.
    By replacing $g$ and $h$ with $g/\lambda$ and $h/\mu$ we may assume $\lambda = \mu = 1$.
    Then $g^3 = h^3 = 1$, and there is a basis in which $g$ and $h$ have the form stated in case \emph{(viii)}.
    
    Alternatively suppose one of $\lambda$ and $\mu$ is not in $K$, say $\lambda$.
    By replacing $h$ with $hg$ or $hg^2$ we may assume $\mu \in K$,
    and then by replacing $h$ with $h / \mu$ we may assume $\mu = 1$ and hence $h^3 = 1$.
    Since $\lambda$ has degree $3$ over $K$, we may identify $K^3$ with $L = \F_{q^3}$ in such a way that $g$ is multiplication by $\lambda$.
    Since $h$ is an $\F_q$-linear map which sends $\lambda \mapsto \omega \lambda \mapsto \omega^2 \lambda \mapsto \lambda$,
    it must be $\frob_q$ or $\frob_q^2$.
    Hence we get case \emph{(ix)}.
\end{proof}

\begin{remark}
	The subgroup structure of $\PSL_3(\F_q)$ was completely determined by Mitchell and Hartley in the early 20th century: see the survey by King~\cite{king-survey}*{Section~2.2}.
	The theorem above is essentially a special case of those classical results (apart from the wrinkle to do with $\PGL$ vs $\PSL$).

	Generally, maximal subgroups of classical groups are classified by a famous theorem of Aschbacher~\cite{aschbacher}. However, maximal abelian subgroups are usually \emph{not} maximal subgroups.
\end{remark}

Via \Cref{cor:restricted-projective-planes}, the subgroups \emph{(i)--(v)} give precisely the five constructions from the introduction (not in that order).
Let us recover, for example, \Cref{csn:1}.
The subgroup \emph{(v)} is
\[
    G = \left\{ \unipotent{a}{a}{b} : a,b \in K \right\} Z/Z \cong
    \begin{cases}
        K^2 &: q~\textup{odd},\\
        C_4^d &: q = 2^d.
    \end{cases}
\]
In \Cref{cor:restricted-projective-planes} take $p = (0 : 0 : 1)$ and $\ell = \{(X : Y : Z) : X = 0\}$.
Then the point--line stabilizer $\{g \in G: p^g \in \ell\}$ is the subgroup defined by $b = 0$.
If $q$ is odd, an isomorphism between $K^2$ and $G$ is given by
\[
    (x, y) \mapsto \unipotent{x}{x}{y + x(x-1)/2}.
\]
Hence a Sidon subset of $K^2$ is defined by $y + x(x-1)/2 = 0$,
which is indeed equivalent to \Cref{csn:1} up to a change of coordinates.
If $q$ is even, we find a Sidon subset of $C_4^d$ of size $2^d$.
An analogous derivation of the other constructions is left as amusement for the reader.

The four other cases \emph{(vi)--(ix)} of \Cref{prop:maximal-abelian} are unproductive from the point of view of \Cref{cor:restricted-projective-planes}.
Indeed, $(vi)$ has no large line orbit while $(vii)$ has no large point orbit, and the other cases are simply too small.

While the previous proposition classifies maximal abelian subgroups of $\PGL_3(K)$,
the fundamental theorem of projective geometry asserts that the full collineation group $\Aut \P^2(K)$
is the projective \emph{semilinear} group
\[
    \PGammaL_3(K) \cong \PGL_3(K) \rtimes \Gal(K),
\]
where $\Gal(K)$ is the Galois group of $K$ over the prime subfield $\F_p$.
There are many further maximal abelian subgroups of $\PGammaL_3(K)$ not contained in $\PGL_3(K)$,
but the following proposition establishes that, like the cases \emph{(vi)--(ix)} of \Cref{prop:maximal-abelian},
they are not useful for \Cref{cor:restricted-projective-planes}.

\begin{proposition}%
  \label{prop:semilinear-are-bad}
  Let $G$ be an abelian subgroup of $\PGammaL_3(K)$ not contained in $\PGL_3(K)$.
  Then $|G| \ll q$.
\end{proposition}
The proof of this proposition is somewhat off-topic so is placed in Appendix~\ref{app:semilinear-bad}.

\section{Nondesarguesian constructions}%
\label{sec:nondesarg}

%\emph{{\bf Freddie}: I will leave it to Sean to determine exactly what is needed here.
%I think it would be good to state at least one explicit non-desarguesian example, assert what the general form of all known examples looks like, and state if possible whether it is conjectured that this is everything.
%Conjectures which merely limit the scope of study should probably be stated very briefly, or confined to references.
%Finally we need everything set up and ready \Cref{cor:conj}.}

We now consider Sidon sets coming from nondesarguesian planes.
Although a dizzying variety of nondesarguesian projective planes are known (see~\cite{weibel} or~\cite{handbook}),
the existence of a large abelian group of collineations cuts down the possibilities considerably,
as established by a fundamental classification theorem of Dembowski and Piper~\cite{dembowski--piper}.
Since we will rely on this theory in the next section, we now briefly summarize what is known and conjectured in this area.

Let $G$ be an abelian (or, more generally, quasiregular) collineation group of a projective plane $\calP$ of order $q$,
and assume $|G| > \frac12 (q^2 + q + 1)$.
Let $t$ be the number of point orbits.
It is possible to show that  $t$ is also the number of line orbits.
Let $F$ be the incidence structure consisting of the fixed points and the fixed lines.
The Dembowski--Piper classification asserts that one of the following holds (the labelling of the cases is standard):
\begin{enumerate}[(a)]
    \item $|G| = q^2 + q + 1$, $t = 1$, and $F$ is empty. In this case $G$ is transitive.
    \item $|G| = q^2$, $t = 3$, and $F$ is a flag, i.e., an incident point-line pair.
    \item $|G| = q^2$, $t = q + 2$, and $F$ is either a line and all its points or dually a point and all its lines.
    \item $|G| = q^2 - 1$, $t = 3$, and $F$ is an antiflag, i.e., a nonincident point-line pair.
    \item $|G| = q^2 - q^{1/2}$, $t = 2$, and $F$ is empty. In this case one point orbit and one line orbit form a subplane of order $q^{1/2}$.
    \item $|G| = q(q-1)$, $t = 5$, and $F$ consists of two points $u$, $v$, the line $\ell$ through $u$ and $v$, and another line $\ell' \neq \ell$ through $v$.
    \item $|G| = {(q-1)}^2$, $t = 7$, and $F$ consits of the vertices and sides of a triangle.
\end{enumerate}
(We have omitted a case included in~\cite{dembowski--piper} that was later shown not to arise in~\cite{ganley-mcfarland}.)

It is conjectured\footnote{This is an amalgamation of several conjectures, including in particular the well-known conjecture that all cyclic projective planes are desarguesian.} that all nondesarguesian planes in the Dembowski--Piper classification are type (b) (see Zhou~\cite{zhou}*{Section~1.9}), and the prime power conjecture for type (b) planes is known~\cite{PPCb}, so from now on assume $q = p^d$ for some prime $p$ and $d \geq 1$.
In even characteristic we must have $G = C_4^d$: see~\cite{zhou}.
In odd characteristic, all known examples have the following special form.

A \emph{planar function} (often called a \emph{perfect nonlinear function} or \emph{bent function} in computer science literature),
introduced by Dembowski and Ostrom~\cite{dembowski--ostrom},
is a function $\phi : \F_q \to \F_q$ such that $x \mapsto \phi(x+h) - \phi(x)$ is a bijection for each $h \neq 0$.
For any planar function $\phi$ the graph
\[
    S = \{(x, \phi(x)) : x \in \F_q\} \subset \F_q^2
\]
is a Sidon set of size $q$ with $\F_q^2 \setminus (S - S) = \{0\} \times \F_q \setminus \{0\}$.

If $\phi$ is quadratic then we get \Cref{csn:1}.
All planar functions over prime fields are quadratic \cites{gluck, ronyai--szonyi, hiramine},
but over general finite fields many nonquadratic examples are known.
Monomial examples include
\begin{align}
    \phi(x) &= x^{p^\alpha+1} & (q = p^d, d/(\alpha,d) ~ \text{odd}), \\
    \phi(x) &= x^{(3^\alpha + 1)/2} & (q = 3^d, (\alpha, 2d) = 1).
\end{align}
The latter example was a breakthrough discovery of Coulter and Matthews~\cite{coulter--matthews}.
There is a conjecture that these are in fact the only monomial examples: see Zieve~\cite{zieve} for progress on this conjecture.

All known examples of planar functions besides the Coulter--Matthews functions have the generalized quadratic form (sometimes called a Dembowski--Ostrom polynomial)
\begin{equation}
    \label{eq:genquad}
    \phi(x) = \sum_{i,j=0}^{d-1} a_{ij} x^{p^i + p^j}.
\end{equation}
It is easily proved that $\phi$ is planar function if and only if the polarization
\[
    \beta(x, y) = \phi(x + y) - \phi(x) - \phi(y) = \sum_{i,j=0}^{d-1} a_{ij} (x^{p^i} y^{p^j} + x^{p^j} y^{p^i})
\]
is nondegenerate in the sense that
\begin{equation}
    \label{eq:nondegen}
    \beta(x, y) = 0 \implies x = 0~\text{or}~y=0.
\end{equation}
For example, following~\cite{coulter--matthews}*{Theorem~3.4}, consider $q = 3^d$, $d$ odd, and
\[
    \phi(x) = x^{10} \pm x^6 - x^2.
\]
The polarization is
\[
    \beta(x, y) = x y^9 + x^9 y \mp x^3 y^3 + xy = xy ((x^4 + y^4)^2 + (x^2 y^2 \pm 1)^2).
\]
Since $a^2 + b^2 = 0$ implies $a = b = 0$ in $K$, the second factor is never zero,
so $\beta$ is nondegenerate and $\phi$ is a planar function.

Classifying planar functions of the form~\eqref{eq:genquad}
is equivalent to classifying symmetric bilinear maps $\beta : \F_p^d \times \F_p^d \to \F_p^d$
satisfying nondegeneracy~\eqref{eq:nondegen},
which in turn is equivalent to classifying commutative semifields up to isotopy.
See Kantor~\cite{kantor-survey} for a slew of examples.

\section{Conjectures}%
\label{sec:conjectures}

Recall that a Sidon subset $S$ of a group $G$ of order $n$ is called \emph{dense} if $|S| \geq (1 - o(1)) n^{1/2}$.
The examples we know (just \Crefrange{csn:1}{csn:5} and the examples in Section~\ref{sec:nondesarg}) point to the following conjecture.

\begin{conjecture}%
  \label{conj:main}
    Suppose $S$ is a dense Sidon set in an abelian group $G$ of order $n$.
    Then $G$ acts faithfully on a projective plane $\calP$ of size $|\calP| = (1+o(1)) n$ in such a way that for some point $p$ and line $\ell$,
    $S \subset \{g \in G : p^g \in \ell\}$.
\end{conjecture}

Equivalently, the development $\dev(S)$ may be completed to a projective plane by adding $o(|G|)$ points and lines.
The following conjecture is slightly weaker.

\begin{conjecture}
    Suppose $S$ is a dense maximal Sidon set in a group $G$ of order $n$.
    Then $T = G \setminus (S - S) \cup 0$ is the union of $O(1)$ subgroups.
\end{conjecture}

For example, in \Cref{csn:5}, $T$ is the union of three subgroups.
It may be that this is the worst case.

We have no idea how to approach these conjectures.
Maybe they are false and we are just bad at constructing examples.
We are unable to solve even the following basic cases.

\begin{conjecture}
    Let $G = \F_p^2$.
    Suppose $S \subset G$ is a Sidon set of size $p$.
    Then $T = G \setminus (S - S) \cup 0$ is a subgroup of order $p$.
\end{conjecture}

Note that if the conclusion above holds then $S$ is a linear transformation of the graph of a function.
By the results cited in Section~\ref{sec:nondesarg}, it then follows that $S$ is a parabola.
This problem was independently posed by Cilleruelo: see~\cite{cilleruelo-mem}*{Problem~3}.

\begin{conjecture}[Michael Tait, personal communication]
    Let $p$ be a (sufficiently large) prime and let $G = C_{p^2+p+1}$.
    Suppose $S \subset G$ is a Sidon set of size $p$.
    Then $S$ is a subset of some perfect difference set $S'$ of size $p+1$.
\end{conjecture}

If true, however, Conjecture~\ref{conj:main} places serious constraints on which abelian groups admit dense Sidon subsets, since we can import the constraints mentioned in Section~\ref{sec:nondesarg}.
One concrete example is the following.

\begin{corollary}%
  \label{cor:conj}
  Suppose \Cref{conj:main} holds.
  Then there is a constant $\eps > 0$ such that the following is true.
  Suppose $S$ is a dense Sidon set in an abelian group $G$ of order $n$, and suppose $|S| > (1-\eps)n^{1/2}$.
  Then
  \[
    |G| \in \{
    q^2+q+1,
    q^2,
    q^2-1,
    q^2-q^{1/2},
    q(q-1),
    (q-1)^2
    \}
  \]
  for some integer $q>1$.
  In particular,
  \[
    \liminf_{n \to \infty} \max \bigl\{ |S| \colon S \subseteq \Z/n\Z \text{ Sidon} \bigr\} / n^{1/2} < 1.
  \]
\end{corollary}
\begin{proof}
    Suppose there is no such constant $\eps>0$.
    Then there is a sequence of abelian groups $G_i$ of order $n_i$ and Sidon sets $S_i \subset G_i$ such that $|S_i| \geq (1 - o(1)) n_i^{1/2}$ and such that $n_i$ is not of any of the given forms.
    Assuming \Cref{conj:main} holds, $G_i$ acts faithfully on a projective plane $\calP_i$ of size $|\calP_i| = (1 + o(1)) n_i$.
    In particular $n_i > |\calP_i| / 2$.
    Now the Dembowski--Piper classification gives a contradiction.
\end{proof}
Further refinements are possible if we also assume some of the conjectures discussed in Section~\ref{sec:nondesarg}.
For example, it should be true that if $G$ admits a dense Sidon set then either it is one of the groups appearing in Proposition~\ref{prop:maximal-abelian}, or $|G|=q^2$ for $q$ a prime power.
Extracting further consequences of this type is left to the reader.

\section{Less dense Sidon sets}%
\label{sec:smaller}

\renewcommand{\theconstruction}{\Alph{construction}}
\setcounter{construction}{0}

\subsection{Background}

As we have mentioned, random or greedy constructions of Sidon sets in a group of order $n$ tend to have size only $n^{1/3}$ or so,
while Sidon sets of size $(1 - o(1))n^{1/2}$ appear to have very restricted structure.
Between these two extremes there is a lot of variety and it is unclear what if any sort of structure should exist in general.

As in the introduction, we collect some examples of these less dense Sidon sets (both published and unpublished) and in some cases their justifications, so that we may later observe a general pattern.
In accordance with Ruzsa's maxim, the constructions all involve the primes in some way.

First, Ruzsa~\cite{ruzsa-infinite} constructed an infinite Sidon set containing $n^{\sqrt2-1+o(1)}$ elements of $\{1, \dots, n\}$ for all $n$.\footnote{Any of the finite constructions (\Crefrange{csn:1}{csn:5}) can be adapted to construction an infinite Sidon set $S \subset \Z$ such that $\limsup |S \cap \{1, \dots, n\}| / n^{1/2} > 0$, but constructing Sidon sets with $|S \cap \{1, \dots, n\}|$ large for \emph{all} $n$ is a different ball game. Despite considerable attention, the exponent $\sqrt{2}-1$ has not been improved.}
  The construction starts with the observation that $\{\log p : p~\text{prime}\}$ is a Sidon set of real numbers.
  As observed by Cilleruelo (see Gowers~\cite{gowers}), a finite version of the same argument produces a Sidon subset of $\{1, \dots, n\}$ of size $n^{1/2} / (\log n)^{3/2}$.

\begin{construction}[Logarithms of primes]%
  \label{construction:log-primes}
  Let $\mathcal{P}_X$ be the set of all primes $p \leq X$, for some parameter $X$ to be determined. For primes $p, q, r, s \in \mathcal{P}_X$, by unique factorization if $\{p, q\} \neq \{r, s\}$ we have $pq \neq rs$, so $|pq - rs| \geq 1$. Since $\log x$ has derivative $1/x$ it follows that
  \[
    |\log(pq) - \log(rs)| \geq X^{-2}.
  \]
  Hence
  \[
    |3X^2 \log p + 3X^2 \log q - 3X^2 \log r - 3X^2 \log s| \geq 3
  \]
  and, defining $\lambda_p = \floor{3X^2 \log p}$,
  \[
    |\lambda_p + \lambda_q - \lambda_r - \lambda_s| \geq 1.
  \]
  Thus $S = \{\lambda_p : p \in \mathcal{P}_X\}$ is a Sidon set in $\{1, \dots, \floor{3X^2 \log X}\}$. Taking $X$ so that $3X^2 \log X \sim n$, we have a Sidon set in $\{1, \dots, n\}$ of size
  \[
    |S| = \pi(X) \sim \frac{X}{\log X} \asymp \frac{n^{1/2}}{(\log n)^{3/2}}.
  \]
\end{construction}

%\subsection{Primes in a quotient ring}%
%\label{subsec:modn-ex}

The next simple example has not appeared much in the literature.
To the best of our knowledge it was first mentioned by Cilleruelo in~\cite{cilleruelo-infinite}.
\begin{construction}[Primes in a quotient ring]%
  \label{construction:quotient-ring}
  Let $m$ be a positive integer and set $G = (\Z/m\Z)^\times$.  The set
  \[
    S = \{ p \bmod m : p \text{ prime},\ 1 < p \le m^{1/2} \}
  \]
  is Sidon: indeed, if four primes $1 < p_1,\dots,p_4 \le m^{1/2}$ obey $p_1 p_2 \equiv p_3 p_4 \pmod m$, i.e.\ $m \mid (p_1 p_2 - p_3 p_4)$, then as $|p_1 p_2 - p_3 p_4| < m$ we must have $p_1 p_2 = p_3 p_4$ and hence $\{p_1, p_2\} = \{p_3, p_4\}$ by unique factorization.  We have
  \[
    |S| = \pi(m^{1/2}) \sim 2 m^{1/2} / \log m
  \]
  and $|G| = \phi(m)$ which is asymptotically somewhere between $m$ and $m/\log \log m$.
\end{construction}

%\subsection{Arguments of Gaussian primes}%
%\label{subsec:gaussian-ex}

The following is a neat variant of Ruzsa's construction.
See Maldonado~\cite{maldonado}*{Theorem~2.2} for details.
\begin{construction}%
  \label{construction:gaussian}
  For each (rational) prime $p \equiv 1 \pmod 4$,
  factorize $p = \rho_p \bar \rho_p$ in $\Z[i]$,
  normalized so that $0 < \Im \rho_p < \Re \rho_p$,
  and let $\phi_p = \arg(\rho_p^4) / 2\pi \in (0, 1/2)$.
  %and define $\phi_p \in (0, 1)$ by $\rho_p / \bar \rho_p = \exp(i 2\pi \phi_p)$.
  Take $S = \{\floor{n \phi_p} : p \leq n^{1/2} / 4\}$.

  The construction achieves $|S| \gg n^{1/2} / \log n$.
\end{construction}

%\subsection{Prime ideal classes}%
%\label{subsec:ideal-ex}

\def\pp{\mathfrak{p}}
\def\mm{\mathfrak{m}}
\def\Q{\mathbf{Q}}
\def\Cl{\opr{Cl}}

The next one was related to us by Ben Green,
who heard it from Ellenberg and Venkatesh.

\begin{construction}%
  \label{construction:imag-quadratic}
  Assume the Generalized Riemann Hypothesis (GRH). %chktex 13
  Let $K = \Q(\sqrt{-D})$ and let $G = \Cl(K)$ be the class group.
  Let $S \subset G$ be a maximal set of prime ideal classes $[\pp]$ with $N\pp < D^{1/4}/2$ having no solutions to $x+y=0$.
  Then
  \begin{align}
      |G| &\leq D^{1/2} (\log D)^{O(1)},\\
      |S| &\geq c D^{1/4} / \log D.
  \end{align}
  Indeed, for each rational prime $p < D^{1/4}/2$ which splits (but does not ramify) in $K$, we may add exactly one of its prime factors $(p) = \pp \bar \pp$ to $S$,
  and we claim that different primes $p$ contribute different classes.
  Indeed, if $p_1, p_2 < D^{1/4}$ and $\pp_1 \mid p_1$, $\pp_2 \mid p_2$ and $\pp_1 \sim \pp_2$, then $\pp_1 \bar \pp_2$ is principal, say $(a + b\sqrt{-D})$, and of norm less than $D^{1/2}/4$, so $b = 0$, so $\pp_1 \bar \pp_2 = (a)$.
  Comparing norms, we have $p_1 \mid a$, hence $\pp_1 \bar \pp_1 = (p_1) \mid (a)$, and by unique factorization we deduce that $\pp_1 = \pp_2$.  
  %Comparing norms, it must be that $a = p_1 = p_2$ and $\pp_1 = \pp_2$.

  By much the same argument we claim $S$ is Sidon.
  Suppose $\pp_i \in S$ ($1 \leq i \leq 4$) satisfy $\pp_i \mid p_i$ and $\pp_1 \pp_2 \sim \pp_3 \pp_4$.
  Then as above, $\pp_1 \pp_2 \bar \pp_3 \bar \pp_4 = (a)$ for some $a \in \Z$.
  Taking norms, we deduce that $p_i | a$ for each $i$ and hence $\pp_i \bar \pp_i | (a)$ for each $i$.
  By unique factorization, it follows that $\pp_1,\pp_2,\bar \pp_3, \bar \pp_4$ can be arranged into two conjugate pairs.
  But since $\pp \in S \Rightarrow \bar \pp \notin S$ by construction, this implies $\{ \pp_1, \pp_2 \} = \{ \pp_3, \pp_4 \}$.
\end{construction}

%Taking norms, $p_1 p_2 p_3 p_4 = a^2$.
%It follows that $\{\pp_1, \pp_2\} = \{\pp_3, \pp_4\}$.
%Hence $S$ is a Sidon set in $G$ of size $|S| \geq |G|^{1/2} / (\log |G|)^{O(1)}$.

%\subsection{Logarithms of primes in real quadratic fields}%
%\label{sub:log-2}
We give a variation of \Cref{construction:log-primes} that is also very similar to \Cref{construction:imag-quadratic}.

\begin{construction}%
  \label{construction:real-quadratic}
  Set $K=\Q(\sqrt{D})$.
  Suppose also that $K$ has class number $1$.\footnote{It is open to show that there are infinitely many such $K$, even on GRH, but in practice this should occur a positive fraction of the time.}
  Let $u \in \cO_K^\times$ be a fundamental unit, and write $r=\log |u| >0$ for the regulator.  Let $M=\lceil r \rceil$.

  Define $S \subset \Z/M\Z$ as follows: for each prime $p$, $1<p \le D^{1/4}/10$ that splits in $K$, factor $p = \pp \bar \pp$ where $\pp =a+b\sqrt{D}$ and $\bar \pp$ denotes the Galois conjugate $a-b\sqrt{D}$.  Then add the element $\lfloor (M/r) \log |\pp/\bar \pp| \rfloor \bmod M$ to $S$.  (Note this definition is unaffected if we change $\pp$ by a unit.)

  We claim that different $p$ give different elements of $S$, as in Section~\ref{construction:imag-quadratic}.
  Note that if $x=a+b \sqrt{D} \in \cO_K$ then either $b=0$ or $|x/\bar x|> D/4N(x)$ or $|x/\bar x| < 4N(x)/D$: indeed, if $a,b>0$ then $|x| > \sqrt{D}/2$ and $|x / \bar x| = |x|^2/N(x)$, and the other cases are analogous.
  Hence, if $1 < p_1, p_2 \le D^{1/4}/10$ and $\bigl|\log |\pp_1 / \bar \pp_1| - \log |\pp_2 / \bar \pp_2| \bigr| < 1$ we set $x=\pp_1 \bar \pp_2$ and obtain a contradiction unless $x \in \Z$, in which case $\pp_1 = \pp_2$ by unique factorization.

  The proof that $S$ is Sidon is by extending this argument in exactly the same way as in Section~\ref{construction:imag-quadratic}, and we omit the details.

  On GRH\footnote{More precisely, the class number formula relates $r\,|\Cl(K)|$ to the residue of $\zeta_K(s)$ at $1$, which is controlled by GRH, and we have assumed $|\Cl(K)|=1$.}, we have $r \le D^{1/2} (\log D)^{O(1)}$, so $S$ is again fairly dense.
\end{construction}

\subsection{A common generalization}%
\label{sub:unified}

We now observe that \Crefrange{construction:log-primes}{construction:real-quadratic} can be placed into a common framework using (essentially) the notion of Hecke characters.

Let $L$ be a number field with integers $\cO_L$.  Write $\sigma_1,\dots,\sigma_r \colon L \to \R$ and $\tau_1,\dots,\tau_s \colon L \to \C$ for its real and complex embeddings, up to isomorphism.
Let $\ell : L^\times \to (\R^\times)^r \times (\C^\times)^s$ be the homomorphism defined by
\[
    \ell(x) = (\sigma_1(x), \dots, \sigma_r(x), \tau_1(x), \dots, \tau_s(x)).
\]
Let $\mm \subset \cO_L$ be an ideal.  Let $I_{\mm}$ denote the abelian group of fractional ideals of $L$ coprime to $\mm$, and for some parameter $R$ let
\[
  \cP_R = \big\{ \pp \in I_{\mm} : \pp \text{ prime},\ N \pp  \le R \big\}.
\]
For a metric abelian group $H$, a group homomorphism $\phi \colon I_{\mm} \to H$ is termed \emph{admissible}\footnote{This condition is natural in the setting of $L$-functions or class field theory.  Its appearance here is motivated only by the previous examples.} if there is a continuous homomorphism
$\psi \colon (\R^\times)^r \times (\C^\times)^s \to H$
such that $\phi\big( (x) \big) = \psi(\ell(x))$ whenever $x \in L^\times$, $x \equiv 1 \pmod{\mm}$, and $\sigma_i(x)>0$ for all $i \in [r]$.
Finally, let $\Lambda \subseteq H$ be a lattice (discrete co-compact subgroup)
and write $[ x ]$ for the nearest point in $\Lambda$ to $x \in H$,
resolving ambiguity in some arbitrary way.
Let
\[
  S = \{ [\phi(\pp)] : \pp \in \cP_R \} \subseteq \Lambda.
\]
If necessary, discard elements from $S$ so that it contains no pair $\{x, -x\}$.
Then $S$ is a Sidon if we can show, for a particular choice of $\mm, R, H, \Lambda$, that
\begin{enumerate}[(i)]
\item $[\phi(\pp_1)] + [\phi(\pp_2)] = [\phi(\pp_3)] + [\phi(\pp_4)]
    \implies
    \phi(\pp_1) + \phi(\pp_2) = \phi(\pp_3) + \phi(\pp_4)$,
\item $\pp_1 \pp_2 \equiv \pp_3 \pp_4 \pmod {\ker \phi}
    \implies \pp_1 \pp_2 = \pp_3 \pp_4$
\end{enumerate}
for all $\pp_1, \pp_2, \pp_3, \pp_4 \in \cP_R$.

Then \Crefrange{construction:log-primes}{construction:real-quadratic} are the following special cases (with minor modifications):
\begin{itemize}
    \item in \Cref{construction:log-primes}, $L=\Q$, $\mm=(1)$, $H = \R$, $\phi( (t) )= \log |t|$, and $\Lambda = 1/(3 R^2) \Z$;
  \item in \Cref{construction:quotient-ring}, $L=\Q$, $\mm=(m)$, $H = \Lambda = (\Z/m\Z)^\times$, $R=m^{1/2}$, and $\phi((x)) = |x| \bmod m$;
    \item in \Cref{construction:gaussian}, $L=\Q(i)$, $\mm=(1)$, $H=\R/\Z$, $\phi((z)) = \arg(z^4) / 2\pi$, and $\Lambda = (1/16 R^2) \Z/\Z$;
    \item in \Cref{construction:imag-quadratic}, $L = \Q(\sqrt{-D})$, $\mm=(1)$, $H = \Lambda = \Cl(L)$, and $\phi$ the quotient map $I_{1} \to \Cl(L)$.
    \item in \Cref{construction:real-quadratic}, $L=\Q(\sqrt{D})$, $\mm=(1)$, $H = \R/r\Z$, $\Lambda = (r/M) \Z/r\Z$ and $\phi$ is the map $(t) \mapsto \log|t / \bar t|$.
\end{itemize}
Here a few ``hybrid'' examples:
\begin{itemize}
    \item Let $L = \Q$, $\mm = (m)$, $H = \R \times (\Z/m\Z)^\times$, and $\phi((x)) = (\log |x|, x \bmod m)$ and $\Lambda = (m/5R^2) \Z \times (\Z/m\Z)^\times$.
      This gives a dense Sidon set by combining the arguments in \Cref{construction:log-primes} and \Cref{construction:quotient-ring}.
    \item Let $L = \Q(\sqrt{-D})$, $\mm=(m)$ for some positive integer $m$, and $H = I_\mm / P_\mm$ (the \emph{ray class group}), where
    $P_{\mm} = \{ (x) : x \in L^\times,\ x \bmod \mm = 1 \}$.
    Let $\phi : I_\mm \to H$ be the quotient map and let $R = D^{1/4} m^{1/2} / 2$.
    This gives a dense Sidon set (on GRH) by combining and extending the observations in \Cref{construction:quotient-ring} and \Cref{construction:imag-quadratic}.
    (Crucially, when writing $\pp_1 \pp_2 \bar \pp_3 \bar \pp_4 = (a + b \sqrt{-D})$, now $m \mid b$, which given $N \pp_i \le R$ forces $b=0$.)
  \item In \Cref{construction:real-quadratic}, we can similarly augment $H$ to $\R/r\Z \times \Cl(K)$ and remove the inconvenient requirement that $|\Cl(K)|=1$. 
\end{itemize}

There is a standard correspondence between Hecke characters and characters on the id\`ele class group,
so it would be equivalent to phrase this construction in terms of id\`eles.

\subsection{Other examples}

We finish by mentioning some further examples of somewhat dense Sidon sets that do not fit the pattern stated above.
A precise classification in general seems hopeless for now, though there are some suggestive analogies.

\begin{construction}
  Let $K$ be a finite field, $\chr K > 3$, let $G = K^2$, and let
  \[
    S = \{(x, x^3) : x \in U\} \subseteq \FF_p^2
  \]
  where $U \subset K$ is some subset.
  One can show that $S$ is a Sidon set if and only if $U$ has at most one solution to $x+y=0$, as in \Cref{construction:imag-quadratic}.
  %Suppose $(x, x^3), (y, y^3), (z, z^3), (w, w^3)$ form an additive quadruple.
  %Then
  %\begin{align}
  %    x + y &= z + w, \\
  %    x^3 + y^3 &= z^3 + w^3.
  %\end{align}
  %Equivalently, subtracting the second equation from the cube of the first,
  %\begin{align}
  %    x + y &= z + w, \\
  %    3xy(x+y) &= 3zw(z + w).
  %\end{align}
  %These equations are clearly satisfied if $x + y = z + w = 0$.
  %On the other hand if $x + y \neq 0$ then we obtain $x + y = z + w$ and $xy = zw$, so it follows that $\{x, y\} = \{z, w\}$.
  %Hence $S$ is a Sidon set if and only if $U$ is a subset of $K$ without more than one solution to $x + y = 0$.
  The largest $S$ can be in this construction is therefore $(q+1)/2$, where $|G| = q^2$.
\end{construction}

%\begin{remark}
%    If $\chr K = 3$ then $S$ is a Sidon set if and only if $U$ is a Sidon set, so $S$ has size at most $O(q^{1/2})$.
%    If $\chr K = 2$ then the function $x \mapsto x^3$ is generalized quadratic in the sense of \Cref{sec:nondesarg}.
%\end{remark}

%\subsection{Jacobian varieties}%
%\label{subsec:jacobian}

\begin{construction}%
  \label{construction:jacobian}
  Recently Forey and Kowalski found a construction involving Jacobian varieties~\cite{FK}.
  Let $K$ be a finite field and $C$ a (hyperelliptic) curve of genus two with a $K$-rational point.
  There is a natural map $\iota$ from $C$ to its Jacobian variety $G$, which is a finite abelian group.
  It can be shown that $S = \iota(C(K))$ is a Sidon set up to a factor of two, similarly to the previous example.
  Moreover, it follows from Weil's proof of the Riemann hypothesis over finite fields that $|S| \sim |G|^{1/2}$,
  so we get a Sidon set of size $\sim |G|^{1/2}/2$ in $A$.
\end{construction}

This is spiritually related to \Cref{construction:imag-quadratic} or \Cref{sub:unified}, under the ``arithmetic geometry'' analogy between ideal/id\`ele class groups and divisor class groups.

%\subsection{Logarithms of primes}%
%\label{subsec:ruzsa-ex}
%
\begin{construction}
  As noted by Gowers~\cite{gowers}, if $S$ is a Sidon subset of $\{1, \dots, n\}$
  then $S' = \{5s + \eps(s) : s \in S\}$ is a Sidon subset of $\{1, \dots, 5n+1\}$,
  where $\eps : S \to \{-1, 0, 1\}$ is arbitrary.
  Thus, if $S$ has some algebraic structure, $S'$ will have somewhat less, and its density will be smaller only by a constant factor.
  As noted by Ben Green, one could also use an irrational multiplier, such as $s \mapsto \floor{s \sqrt{2}}$.
\end{construction}

\appendix
\section{Proof of Proposition~\ref{prop:semilinear-are-bad}}%
\label{app:semilinear-bad}

Recall we wish to show that if $K = \F_q$ and $H \le \PGammaL_3(K)$ is an abelian subgroup not contained in $\PGL_3(K)$ then $|H| \ll q$.

Let $G$ be the subgroup of $\GammaL_3(K) = \GL_3(K) \rtimes \Gal(K)$ upstairs corresponding to $H$. Hence $Z \le G$, $G$ is not contained in $\GL_3(K)$, $G' \leq Z$, and $|G| = (q-1) |H|$. We wish to show that $|G| \ll q^2$.

Let $G_0 = G \cap \GL_3(K)$.  Fix an element $h = g \sigma \in G$ such that $\sigma \in \Gal(K)$ generates the (nontrivial) image of $G$ in the cyclic group $\Gal(K)$. Let $k \subset K$ denote the fixed field of $\sigma$, and write $d = [K : k]$ for the degree of $K$ over $k$; equivalently, $d$ is the order of $\sigma$ in $\Gal(K)$.  Hence, $|G| = d\, |G_0|$.

If $G_0$ is nonabelian, it is a subgroup of either \emph{(viii)} or \emph{(ix)} in Proposition~\ref{prop:maximal-abelian}, so $|G_0| \ll q$, $|G| \ll d q \le q \log_p q$ and we are done.  We now assume that $G_0$ is abelian.

Consider the function $\lambda \colon G_0 \to K^\times$ defined by $[h, a] = hah^{-1}a^{-1} = \lambda(a) I$.  We observe it is a group homomorphism.  Moreover,
\[
  \lambda(a) a = h a h^{-1} = g \sigma(a) g^{-1}
\]
and hence
\[
  \lambda(a)^3\, (\det a) = \det \sigma(a) = \sigma(\det a)
\]
so $\lambda(a)^3$ has the form $\sigma(t)/t$ for some $t \in K^\times$, hence $N_{K/k}(\lambda(a)^3) = 1$.  There are $|K^\times|/|k^\times|$ elements $u \in K^{\times}$ with $N_{K/k}(u)=1$, so $|\lambda(G_0)| \le 3 |K^{\times}|/|k^{\times}|$.

Let $G_{00} = \ker \lambda \le G_0$.  By definition, for all $a \in G_{00}$ we have
\begin{equation}%
  \label{eq:galois-action}
  \sigma(a) = g^{-1} a g.
\end{equation}
Let $A_k = \spn_k(G_{00}) \le M_3(K)$ and $A_K = \spn_K(G_{00}) \le M_3(K)$.  Note $A_K$ is a commutative $K$-subalgebra of $M_3(K)$, and by a theorem of Schur~\cites{schur,mirzahkani}%
\footnote{Alternatively, in dimension $3$, this could be extracted from Proposition~\ref{prop:maximal-abelian} or its proof.}
any commutative $K$-subalgebra of $M_3(K)$ has $K$-dimension at most $3$, so $\dim_K(A_K) \le 3$.

We claim $\dim_k(A_k) \le 3$.  Indeed, any $a_1,a_2,a_3,a_4 \in A_k$ must be linearly dependent over $K$ (as they lie in $A_K$): say $\sum_{i=1}^4 t_i a_i = 0$ for some $t_i \in K$, not all zero.  Applying~\eqref{eq:galois-action},
\[
  \sum_{i=1}^4 \sigma(t_i)\, a_i = 0
\]
and by iterating this and summing, we obtain
\[
  \sum_{i=1}^4 \tr_{K/k}(t_i)\, a_i = 0.
\]
Finally, we may apply this replacing $t_i$ with $u t_i$ for any $u \in K$ throughout, and $u$ may be chosen so that some $\tr_{K/k}(u t_i)$ is non-zero.  Hence $a_1,\dots,a_4$ are necessarily linearly dependent over $k$, so $\dim_k(A_k) \le 3$ as claimed.

It follows that $|G_{00}| \le |k|^3 - 1$ (as $0 \in A_k$).  Putting everything together, we deduce
\[
  |G| \le 3 d \frac{|K^\times|}{|k^\times|} (|k|^3 - 1) = 3 d (q-1) (q^{2/d} + q^{1/d} + 1).
\]
When $d=2$ this is $O(q^2)$ as claimed,%
\footnote{In this case the above bound can be sharp, at least up to the factor of $3$. For example, suppose $q = p^2$, $h$ is $\frob_{p^3}$, and $G_0 = K^\times \F_{p^3}^\times$.}
and for $3 \le d \le \log_p q$ this implies an even stronger bound.

\bibliography{refs}

\begin{bibdiv}
\begin{biblist}

\bib{aschbacher}{article}{
      author={Aschbacher, M.},
       title={On the maximal subgroups of the finite classical groups},
        date={1984},
        ISSN={0020-9910},
     journal={Invent. Math.},
      volume={76},
      number={3},
       pages={469\ndash 514},
         url={https://doi.org/10.1007/BF01388470},
      review={\MR{746539}},
}

\bib{PPCb}{article}{
      author={Blokhuis, Aart},
      author={Jungnickel, Dieter},
      author={Schmidt, Bernhard},
       title={Proof of the prime power conjecture for projective planes of
  order {$n$} with abelian collineation groups of order {$n^2$}},
        date={2002},
        ISSN={0002-9939},
     journal={Proc. Amer. Math. Soc.},
      volume={130},
      number={5},
       pages={1473\ndash 1476},
         url={https://doi.org/10.1090/S0002-9939-01-06388-2},
      review={\MR{1879972}},
}

\bib{bose}{article}{
      author={Bose, R.~C.},
       title={An affine analogue of {S}inger's theorem},
        date={1942},
        ISSN={0019-5839},
     journal={J. Indian Math. Soc. (N.S.)},
      volume={6},
       pages={1\ndash 15},
      review={\MR{6735}},
}

\bib{cilleruelo}{article}{
      author={Cilleruelo, Javier},
       title={Combinatorial problems in finite fields and {S}idon sets},
        date={2012},
        ISSN={0209-9683},
     journal={Combinatorica},
      volume={32},
      number={5},
       pages={497\ndash 511},
         url={https://doi.org/10.1007/s00493-012-2819-4},
      review={\MR{3004806}},
}

\bib{cilleruelo-infinite}{article}{
      author={Cilleruelo, Javier},
       title={Infinite {S}idon sequences},
        date={2014},
        ISSN={0001-8708},
     journal={Adv. Math.},
      volume={255},
       pages={474\ndash 486},
         url={https://doi.org/10.1016/j.aim.2014.01.011},
      review={\MR{3167490}},
}

\bib{coulter--matthews}{article}{
      author={Coulter, Robert~S.},
      author={Matthews, Rex~W.},
       title={Planar functions and planes of {L}enz-{B}arlotti class {II}},
        date={1997},
        ISSN={0925-1022},
     journal={Des. Codes Cryptogr.},
      volume={10},
      number={2},
       pages={167\ndash 184},
         url={https://doi.org/10.1023/A:1008292303803},
      review={\MR{1432296}},
}

\bib{cilleruelo-mem}{article}{
      author={Candela, Pablo},
      author={Ru\'{e}, Juanjo},
      author={Serra, Oriol},
       title={Memorial to {J}avier {C}illeruelo: a problem list},
        date={2018},
     journal={Integers},
      volume={18},
       pages={Paper No. A28, 9},
      review={\MR{3783887}},
}

\bib{dembowski--ostrom}{article}{
      author={Dembowski, Peter},
      author={Ostrom, T.~G.},
       title={Planes of order {$n$} with collineation groups of order
  {$n^{2}$}},
        date={1968},
        ISSN={0025-5874},
     journal={Math. Z.},
      volume={103},
       pages={239\ndash 258},
         url={https://doi.org/10.1007/BF01111042},
      review={\MR{226486}},
}

\bib{dembowski--piper}{article}{
      author={Dembowski, Peter},
      author={Piper, Fred},
       title={Quasiregular collineation groups of finite projective planes},
        date={1967},
        ISSN={0025-5874},
     journal={Math. Z.},
      volume={99},
       pages={53\ndash 75},
         url={https://doi.org/10.1007/BF01118689},
      review={\MR{215741}},
}

\bib{erdos--turan}{article}{
      author={Erd\"{o}s, P.},
      author={Tur\'{a}n, P.},
       title={On a problem of {S}idon in additive number theory, and on some
  related problems},
        date={1941},
        ISSN={0024-6107},
     journal={J. London Math. Soc.},
      volume={16},
       pages={212\ndash 215},
         url={https://doi.org/10.1112/jlms/s1-16.4.212},
      review={\MR{6197}},
}

\bib{FK}{article}{
      author={{Forey}, Arthur},
      author={{Kowalski}, Emmanuel},
       title={{Algebraic curves in their Jacobian are Sidon sets}},
        date={2021-03},
     journal={arXiv e-prints},
       pages={arXiv:2103.04917},
      eprint={2103.04917},
}

\bib{ganley}{article}{
      author={Ganley, Michael~J.},
       title={Direct product difference sets},
        date={1977},
        ISSN={0097-3165},
     journal={J. Combinatorial Theory Ser. A},
      volume={23},
      number={3},
       pages={321\ndash 332},
         url={https://doi.org/10.1016/0097-3165(77)90023-1},
      review={\MR{480082}},
}

\bib{gluck}{article}{
      author={Gluck, David},
       title={A note on permutation polynomials and finite geometries},
        date={1990},
        ISSN={0012-365X},
     journal={Discrete Math.},
      volume={80},
      number={1},
       pages={97\ndash 100},
         url={https://doi.org/10.1016/0012-365X(90)90299-W},
      review={\MR{1045927}},
}

\bib{ganley-mcfarland}{article}{
      author={Ganley, Michael~J.},
      author={McFarland, Robert~L.},
       title={On quasiregular collineation groups},
        date={1975},
        ISSN={0003-889X},
     journal={Arch. Math. (Basel)},
      volume={26},
       pages={327\ndash 331},
         url={https://doi.org/10.1007/BF01229747},
      review={\MR{385703}},
}

\bib{gowers}{misc}{
      author={Gowers, Tim},
       title={What are dense sidon subsets of $\{1, \dots, n\}$ like?},
        date={2012},
  note={https://gowers.wordpress.com/2012/07/13/what-are-dense-sidon-subsets-of-12-n-like/},
}

\bib{hiramine}{article}{
      author={Hiramine, Yutaka},
       title={A conjecture on affine planes of prime order},
        date={1989},
        ISSN={0097-3165},
     journal={J. Combin. Theory Ser. A},
      volume={52},
      number={1},
       pages={44\ndash 50},
         url={https://doi.org/10.1016/0097-3165(89)90060-5},
      review={\MR{1008158}},
}

\bib{hughes}{article}{
      author={Hughes, D.~R.},
       title={Planar division neo-rings},
        date={1955},
        ISSN={0002-9947},
     journal={Trans. Amer. Math. Soc.},
      volume={80},
       pages={502\ndash 527},
         url={https://doi.org/10.2307/1993000},
      review={\MR{73566}},
}

\bib{handbook}{book}{
      author={Johnson, Norman~L.},
      author={Jha, Vikram},
      author={Biliotti, Mauro},
       title={Handbook of finite translation planes},
      series={Pure and Applied Mathematics (Boca Raton)},
   publisher={Chapman \& Hall/CRC, Boca Raton, FL},
        date={2007},
      volume={289},
        ISBN={978-1-58488-605-1; 1-58488-605-6},
         url={https://doi.org/10.1201/9781420011142},
      review={\MR{2290291}},
}

\bib{kantor-survey}{incollection}{
      author={Kantor, William~M.},
       title={Finite semifields},
        date={2006},
   booktitle={Finite geometries, groups, and computation},
   publisher={Walter de Gruyter, Berlin},
       pages={103\ndash 114},
      review={\MR{2258004}},
}

\bib{king-survey}{incollection}{
      author={King, Oliver~H.},
       title={The subgroup structure of finite classical groups in terms of
  geometric configurations},
        date={2005},
   booktitle={Surveys in combinatorics 2005},
      series={London Math. Soc. Lecture Note Ser.},
      volume={327},
   publisher={Cambridge Univ. Press, Cambridge},
       pages={29\ndash 56},
         url={https://doi.org/10.1017/CBO9780511734885.003},
      review={\MR{2187733}},
}

\bib{mirzahkani}{article}{
      author={Mirzakhani, M.},
       title={A simple proof of a theorem of {S}chur},
        date={1998},
        ISSN={0002-9890},
     journal={Amer. Math. Monthly},
      volume={105},
      number={3},
       pages={260\ndash 262},
         url={https://doi.org/10.2307/2589084},
      review={\MR{1615548}},
}

\bib{maldonado}{article}{
      author={Maldonado~L\'{o}pez, Juan~Pablo},
       title={A remark on infinite {S}idon sets},
        date={2011},
        ISSN={0034-7426},
     journal={Rev. Colombiana Mat.},
      volume={45},
      number={2},
       pages={113\ndash 127},
      review={\MR{2925311}},
}

\bib{ronyai--szonyi}{article}{
      author={R\'{o}nyai, L.},
      author={Sz\H{o}nyi, T.},
       title={Planar functions over finite fields},
        date={1989},
        ISSN={0209-9683},
     journal={Combinatorica},
      volume={9},
      number={3},
       pages={315\ndash 320},
         url={https://doi.org/10.1007/BF02125898},
      review={\MR{1030384}},
}

\bib{ruzsa}{article}{
      author={Ruzsa, Imre~Z.},
       title={Solving a linear equation in a set of integers. {I}},
        date={1993},
        ISSN={0065-1036},
     journal={Acta Arith.},
      volume={65},
      number={3},
       pages={259\ndash 282},
         url={https://doi.org/10.4064/aa-65-3-259-282},
      review={\MR{1254961}},
}

\bib{ruzsa-infinite}{article}{
      author={Ruzsa, Imre~Z.},
       title={An infinite {S}idon sequence},
        date={1998},
        ISSN={0022-314X},
     journal={J. Number Theory},
      volume={68},
      number={1},
       pages={63\ndash 71},
         url={https://doi.org/10.1006/jnth.1997.2192},
      review={\MR{1492889}},
}

\bib{ruzsa-survey}{article}{
      author={Ruzsa, Imre~Z.},
       title={Erd{\H{o}}s and the integers},
        date={1999},
        ISSN={0022-314X},
     journal={J. Number Theory},
      volume={79},
      number={1},
       pages={115\ndash 163},
         url={https://doi.org/10.1006/jnth.1999.2395},
      review={\MR{1724257}},
}

\bib{schur}{article}{
      author={Schur, J.},
       title={Zur {T}heorie der vertauschbaren {M}atrizen},
        date={1905},
        ISSN={0075-4102},
     journal={J. Reine Angew. Math.},
      volume={130},
       pages={66\ndash 76},
         url={https://doi.org/10.1515/crll.1905.130.66},
      review={\MR{1580676}},
}

\bib{singer}{article}{
      author={Singer, James},
       title={A theorem in finite projective geometry and some applications to
  number theory},
        date={1938},
        ISSN={0002-9947},
     journal={Trans. Amer. Math. Soc.},
      volume={43},
      number={3},
       pages={377\ndash 385},
         url={https://doi.org/10.2307/1990067},
      review={\MR{1501951}},
}

\bib{tait-timmons}{article}{
      author={Tait, Michael},
      author={Timmons, Craig},
       title={Orthogonal polarity graphs and {S}idon sets},
        date={2016},
        ISSN={0364-9024},
     journal={J. Graph Theory},
      volume={82},
      number={1},
       pages={103\ndash 116},
         url={https://doi.org/10.1002/jgt.21890},
      review={\MR{3477036}},
}

\bib{tao--vu}{book}{
      author={Tao, Terence},
      author={Vu, Van~H.},
       title={Additive combinatorics},
      series={Cambridge Studies in Advanced Mathematics},
   publisher={Cambridge University Press, Cambridge},
        date={2010},
      volume={105},
        ISBN={978-0-521-13656-3},
        note={Paperback edition [of MR2289012]},
      review={\MR{2573797}},
}

\bib{wehrfritz}{book}{
      author={Wehrfritz, B. A.~F.},
       title={Infinite linear groups. {A}n account of the group-theoretic
  properties of infinite groups of matrices},
   publisher={Springer-Verlag, New York-Heidelberg},
        date={1973},
        note={Ergebnisse der Matematik und ihrer Grenzgebiete, Band 76},
      review={\MR{0335656}},
}

\bib{weibel}{inproceedings}{
      author={Weibel, Charles},
       title={Survey of non-desarguesian planes},
        date={2007},
       pages={1294\ndash 1303},
}

\bib{zhou}{thesis}{
      author={Zhou, Yue},
       title={Difference sets from projective planes},
        type={Ph.D. Thesis},
        date={2013},
         url={http://dx.doi.org/10.25673/3955},
}

\bib{zieve}{article}{
      author={Zieve, Michael~E.},
       title={Planar functions and perfect nonlinear monomials over finite
  fields},
        date={2015},
        ISSN={0925-1022},
     journal={Des. Codes Cryptogr.},
      volume={75},
      number={1},
       pages={71\ndash 80},
         url={https://doi.org/10.1007/s10623-013-9890-8},
      review={\MR{3320352}},
}

\end{biblist}
\end{bibdiv}
\bibliographystyle{alpha}
\end{document}